\documentclass[a4paper,12pt,leqno]{article}
\usepackage[margin=3cm]{geometry}
\usepackage[all,2cell]{xy}
\UseAllTwocells
\usepackage{stackrel}
\usepackage{amsmath}
\usepackage{amssymb}
\usepackage{amsxtra}
\usepackage{amsthm}
\usepackage{colonequals}
\usepackage[T1]{fontenc}
\usepackage{lmodern}
\usepackage[pagebackref,breaklinks]{hyperref}
\usepackage[abbrev,shortalphabetic]{amsrefs}
\usepackage[shortlabels]{enumitem}
\setlist[enumerate,1]{(a)}
\setlist{nosep}

\theoremstyle{definition}
\newtheorem{definition}{Definition}[section]
\newtheorem{notation}[definition]{Notation}

\newtheorem{construction}[definition]{Construction}
\theoremstyle{plain}
\newtheorem{prop}[definition]{Proposition}
\newtheorem{lemma}[definition]{Lemma}
\newtheorem{theorem}[definition]{Theorem}
\newtheorem{cor}[definition]{Corollary}
\theoremstyle{remark}
\newtheorem{remark}[definition]{Remark}

\newcommand{\Z}{\mathbb{Z}}

\newcommand{\cB}{\mathcal{B}}
\newcommand{\cC}{\mathcal{C}}
\newcommand{\Cat}{\mathcal{C}\mathnormal{at}}
\newcommand{\cD}{\mathcal{D}}

\newcommand{\Spec}{\mathrm{Spec}}

\newcommand{\End}{\mathrm{End}}
\newcommand{\Hom}{\mathrm{Hom}}
\newcommand{\id}{\mathrm{id}}
\newcommand{\co}{\mathrm{co}}
\newcommand{\adj}{\mathrm{adj}}
\newcommand{\cc}{\mathrm{cc}}

\newcommand{\spe}{\mathrm{sp}}
\newcommand{\ev}{\mathrm{ev}}
\newcommand{\coev}{\mathrm{coev}}
\newcommand{\ft}{{\mathrm{ft}}}
\newcommand{\cft}{{c\mathrm{ft}}}
\newcommand{\tr}{\mathrm{tr}}

\newcommand{\cHom}{\mathcal{H}\mathnormal{om}}
\newcommand{\vecl}[1]{\overleftarrow{#1}}
\newcommand{\vecr}[1]{\overrightarrow{#1}}
\newcommand{\atimes}{\stackbin{\leftarrow}{\times}}
\newcommand{\ttimes}{\mathbin{\bar{\times}}}
\newcommand{\aPsi}{\overleftarrow{\Psi}}
\numberwithin{equation}{section}

\begin{document}
\title{Categorical traces and a relative Lefschetz-Verdier formula}
\author{Qing Lu\thanks{School of Mathematical Sciences, Beijing Normal University, Beijing
100875, China; email: \texttt{qlu@bnu.edu.cn}.} \and Weizhe
Zheng\thanks{Morningside Center of Mathematics, Academy of Mathematics and
Systems Science, Chinese Academy of Sciences, Beijing 100190, China;
University of the Chinese Academy of Sciences, Beijing 100049, China; email:
\texttt{wzheng@math.ac.cn}.}}
\date{}
\maketitle

\begin{abstract}
We prove a relative Lefschetz-Verdier theorem for locally acyclic objects
over a Noetherian base scheme. This is done by studying duals and traces
in the symmetric monoidal $2$-category of cohomological correspondences.
We show that local acyclicity is equivalent to dualizability and deduce
that duality preserves local acyclicity. As another application of the
category of cohomological correspondences, we show that the nearby cycle
functor over a Henselian valuation ring preserves duals, generalizing a
theorem of Gabber.
\end{abstract}

\section*{Introduction}
The notions of dual and trace in symmetric monoidal categories were
introduced by Dold and Puppe \cite{DP}. They have been extended to higher
categories and have found important applications in algebraic geometry and
other contexts (see \cite{BN} by Ben-Zvi and Nadler and the references
therein).

The goal of the present paper is to record several applications of the
formalism of duals and traces to the symmetric monoidal $2$-category of
cohomological correspondences in \'etale cohomology. One of our main results
is the following relative Lefschetz-Verdier theorem.

\begin{theorem}\label{t.lv0}
Let $S$ be a Noetherian scheme and let $\Lambda$ be a Noetherian commutative
ring with $m\Lambda=0$ for some $m$ invertible on $S$. Let
\[\xymatrix{X\ar[d]_{f} & C\ar[l]_{\vecl{c}}\ar[d]^{p}\ar[r]^{\vecr{c}} & Y\ar[d]^{g} & D\ar[l]_{\vecl{d}}\ar[r]^{\vecr{d}}\ar[d]^{q} &
X\ar[d]^{f}\\
X' & C'\ar[l]\ar[r] & Y' & D'\ar[l]\ar[r] & X'}
\]
be a commutative diagram of schemes separated of finite type over $S$, with
$p$ and $D\to D'\times_{Y'} Y$ proper. Let $L\in D_{\cft}(X,\Lambda)$ such
that $L$ and $f_!L$ are locally acyclic over $S$. Let $M\in D(Y,\Lambda)$,
$u\colon \vecl{c}^*L\to \vecr{c}^! M$, $v\colon \vecl{d}^*M\to \vecr{d}^!
L$. Then $s\colon C\times_{X\times_S Y} D\to C'\times_{X'\times_S Y'} D'$ is
proper and
\[s_*\langle u,v\rangle = \langle (f,p,g)_!u,(g,q,f)_!v\rangle.\]
\end{theorem}

Here $D_\cft(X,\Lambda)\subseteq D(X,\Lambda)$ denotes the full subcategory
spanned by objects of finite tor-dimension and of constructible cohomology
sheaves, and $\langle u,v\rangle$ is the relative Lefschetz-Verdier pairing.

\begin{remark}
In the case where $S$ is the spectrum of a field, local acyclicity is
trivial and the theorem generalizes \cite[III Corollaire 4.5]{SGA5} and (the
scheme case of) \cite[Proposition 1.2.5]{Var}. For $S$ smooth over a perfect
field and under additional assumptions of smoothness and transversality,
Theorem \ref{t.lv0} was proved by Yang and Zhao \cite[Corollary 3.10]{YZ}.
The original proof in \cite{SGA5} and its adaptation in \cite{YZ} require
the verification of a large amount of commutative diagrams. The categorical
interpretation we adopt makes our proof arguably more conceptual.
\end{remark}

It was observed by Lurie that Grothendieck's cohomological operations can be
encoded by a (pseudo) functor $\cB\to \Cat$, where $\cB$ denotes the
category of correspondences and $\Cat$ denotes the $2$-category of
categories. Contrary to the situation of \cite[Definition 2.15]{BN}, in the
context of \'etale cohomology, the functor has a \emph{right-lax} symmetric
monoidal structure that is not expected to be symmetric monoidal even after
enhancement to higher categories. Instead, we apply the formalism of traces
to the corresponding cofibered category produced by the Grothendieck
construction, which is the category $\cC$ of cohomological correspondences.
The relative Lefschetz-Verdier formula follows from the functoriality of
traces for dualizable objects $(X,L)$ of $\cC$.

To complete the proof, we show that under the assumption $L\in
D_{\cft}(X,\Lambda)$, dualizability is equivalent to local acyclicity
(Theorem \ref{t.dual}). As a byproduct of this equivalence, we deduce
immediately that local acyclicity is preserved by duality (Corollary
\ref{c.la}). Note that this last statement does not involve cohomological
correspondences.

We also give applications to the nearby cycle functor $\Psi$ over a
Henselian valuation ring. The functor $\Psi$ extends the usual nearby cycle
functor over a Henselian \emph{discrete} valuation ring and was studied by
Huber \cite[Section 4.2]{Huber}. By studying specialization of cohomological
correspondences, we generalize Gabber's theorem that $\Psi$ preserves duals
and a fixed point theorem of Vidal to Henselian valuation rings (Corollaries
\ref{c.Gabber} and \ref{c.Vidal}). We hope that the latter can be used to
study ramification over higher-dimensional bases.

Scholze remarked that our arguments also apply in the \'etale cohomology of
diamonds and imply the equivalence between dualizability and universal local
acyclicity in this situation. This fact and applications are discussed in
his work with Fargues on the geometrization of the Langlands correspondence
\cite{FS}.

Let us briefly mention some other categorical approaches to Lefschetz type
theorems. In \cite[Section~4]{DP}, the Lefschetz fixed point theorem is
deduced from the functoriality of traces by passing to suspension spectra.
In \cite{Petit}, a categorical framework is set up for Lefschetz-Lunts type
formulas. In May 2019, as a first draft of this paper was being written,
Varshavsky informed us that he had a different strategy to deduce the
Lefschetz-Verdier formula, using categorical traces in
$(\infty,2)$-categories.

This paper is organized as follows. In Section \ref{s.1}, we review duals
and traces in symmetric monoidal $2$-categories and the Grothendieck
construction. In Section \ref{s.2}, we define the symmetric monoidal
$2$-category of cohomological correspondences and prove the relative
Lefschetz-Verdier theorem. In Section \ref{s.3}, we discuss applications to
the nearby cycle functor over a Henselian valuation ring.

\subsection*{Acknowledgments}
We learned categorical traces from lectures by Xinwen Zhu and David Nadler.
We would like to thank Ning Guo, Luc Illusie, Yifeng Liu, Fran\c cois Petit,
Takeshi Saito, Peter Scholze, Yakov Varshavsky, Cong Xue, Enlin Yang, and
Xinwen Zhu for discussions and comments. We are grateful for the support of
Princeton University, where part of this work was done during a visit.

\section{Pairings in symmetric monoidal $2$-categories}\label{s.1}
We review duals, traces, and pairings in symmetric monoidal $2$-categories.
We give the definitions in Subsection \ref{ss.1} and discuss the
functoriality of pairings in Subsection \ref{ss.2}. These two subsections
are mostly standard (see \cite{BN} and \cite{HSS} for generalizations to
higher categories). In Subsection \ref{ss.3} we review the Grothendieck
construction in the symmetric monoidal context, which will be used to
interpret the category of cohomological correspondences later.

By a $2$-category, we mean a weak $2$-category (also known as a bicategory
in the literature).

\subsection{Pairings}\label{ss.1}
Let $(\cC,\otimes,1_{\cC})$ be a symmetric monoidal $2$-category.

\begin{definition}[dual]\label{d.dual}
An object $X$ of $\cC$ is \emph{dualizable} if there exist an object
$X\spcheck$ of~$\cC$, called the \emph{dual} of $X$, and morphisms
$\ev_X\colon X\spcheck\otimes X\to 1_\cC$, $\coev_X\colon 1_\cC\to X\otimes
X\spcheck$, called evaluation and coevaluation, respectively, such that the
composites
\[X\xrightarrow{\coev_X\otimes \id_X} X\otimes X\spcheck \otimes X \xrightarrow{\id_X\otimes \ev_X} X,\quad X\spcheck\xrightarrow{\id_{X\spcheck}\otimes \coev_X} X\spcheck \otimes X\otimes X\spcheck \xrightarrow{\ev_X\otimes \id_{X\spcheck}} X\spcheck\]
are isomorphic to identities.
\end{definition}

\begin{remark}\label{r.dual}
For $X$ dualizable, $X\spcheck$ is dualizable of dual $X$. For $X$ and $Y$
dualizable, $X\otimes Y$ is dualizable of dual $X\spcheck\otimes Y\spcheck$.
\end{remark}

For $X$ and $Y$ in $\cC$, we let $\cHom(X,Y)$ denote the internal mapping
object if it exists.

\begin{remark}\label{r.dual2}
Assume that $X$ is dualizable of dual $X\spcheck$.
\begin{enumerate}
\item The morphisms $\coev_X$ and $\ev_X$ exhibit $-\otimes X\spcheck$ as
    right (and left) adjoint to $-\otimes X$. Thus, for every object $Y$,
    $\cHom(X,Y)$ exists and is equivalent to $Y\otimes X\spcheck$. In
    particular, $\cHom(X,1_\cC)$ exists and is equivalent to $X\spcheck$.

\item If, moreover, $\cHom(Y,1_\cC)$ exists, then we have equivalences
\begin{gather*}
\cHom(X\otimes
Y,1_\cC)\simeq \cHom(X,\cHom(Y,1_\cC))\stackrel{\text{(a)}}{\simeq}
\cHom(Y,1_\cC)\otimes \cHom(X,1_\cC),\\
\cHom(Y,X)\simeq \cHom(Y,\cHom(X\spcheck,1_\cC))\simeq \cHom(X\spcheck\otimes Y,1_\cC)\qquad\\
\qquad\simeq \cHom(Y,1_\cC)\otimes\cHom(X\spcheck,1_\cC)\simeq \cHom(Y,1_\cC)\otimes X.
\end{gather*}
\end{enumerate}
\end{remark}

\begin{lemma}\label{l.Hom0}
An object $X$ is dualizable if and only if $\cHom(X,1_\cC)$ and $\cHom(X,X)$
exist and the morphism $m\colon X\otimes \cHom(X,1_\cC)\to \cHom(X,X)$
adjoint to
\[
X\otimes \cHom(X,1_\cC)\otimes X\xrightarrow{\id_X\otimes \ev_X} X
\]
is a split epimorphism. Here $\ev_X\colon \cHom(X,1_\cC)\otimes X\to 1_\cC$
denotes the counit.
\end{lemma}

\begin{proof}
The ``only if'' part is a special case of Remark \ref{r.dual2}. For the
``if'' part, we define $\coev_X\colon 1_\cC\to X\otimes \cHom(X,1_\cC)$ to
be the composite of a section of $m$ and the morphism $1_\cC\to \cHom(X,X)$
corresponding to $\id_X$. It is easy to see that $\ev_X$ and $\coev_X$
exhibit $\cHom(X,1_\cC)$ as a dual of $X$.
\end{proof}

For $X$ and $Y$ dualizable, the dual of a morphism $u\colon X\to Y$ is the
composite
\[u\spcheck\colon Y\spcheck\xrightarrow{\id_{Y\spcheck}\otimes \coev_X} Y\spcheck \otimes X\otimes X\spcheck \xrightarrow{\id_{Y\spcheck}\otimes u\otimes \id_{X\spcheck}} Y\spcheck \otimes Y\otimes X\spcheck \xrightarrow{\ev_Y\otimes \id_{X\spcheck}} X\spcheck.\]
This construction gives rise to a functor $\Hom_{\cC}(X,Y)\to
\Hom_{\cC}(Y\spcheck,X\spcheck)$. We have commutative squares with
invertible $2$-morphisms
\begin{equation}\label{e.ucheck}
\xymatrix{1_\cC\ar[r]^{\coev_X}\ar[d]_{\coev_Y} & X\otimes X\spcheck \ar[d]^{u\otimes \id}
& X\otimes Y\spcheck\ar[r]^{u\otimes \id}\ar[d]_{\id\otimes u\spcheck} & Y\otimes Y\spcheck\ar[d]^{\ev_Y}\\
Y\otimes Y\spcheck \ar[r]^{\id\otimes u\spcheck} & Y\otimes X\spcheck & X\otimes X\spcheck\ar[r]^{\ev_X} & 1_\cC.}
\end{equation}
Moreover, for $X\xrightarrow{u} Y\xrightarrow{v} Z$ with $X$, $Y$, $Z$
dualizable, we have $(vu)\spcheck\simeq u\spcheck v\spcheck$.

\begin{notation}
We let $\Omega \cC$ denote the category $\End(1_\cC)$.
\end{notation}

\begin{construction}[dimension, trace, and pairing]\label{c.pair}
Let $X$ be a dualizable object of $\cC$ and let $e\colon X\to X$ be an
endomorphism.  We define the trace $\tr(e)$ to be the object of $\Omega \cC$
given by the composite
\[1_\cC\xrightarrow{\coev_X} X\otimes X\spcheck \xrightarrow{e\otimes \id_{X\spcheck}} X\otimes X\spcheck \xrightarrow{\ev_X} 1_\cC,\]
where in the last arrow we used the commutativity constraint.

Let $u\colon X\to Y$ and $v\colon Y\to X$ be morphisms with $X$ dualizable.
We define the pairing by $\langle u,v\rangle=\tr(v\circ u)$.

We define the \emph{dimension} of a dualizable object $X$ to be
$\dim(X)\colonequals \langle \id_X,\id_X\rangle$, which is the composite
$1_\cC\xrightarrow{\coev_X} X\otimes X\spcheck \xrightarrow{\ev_X} 1_\cC$.
\end{construction}

If $X$ and $Y$ are both dualizable, then $\langle u,v\rangle$ is isomorphic
to the composite
\[1_\cC\xrightarrow{\coev_X} X\otimes X\spcheck \xrightarrow{u\otimes v\spcheck} Y\otimes Y\spcheck \xrightarrow{\ev_Y} 1_\cC.\]
In this case, we have an isomorphism $\langle u,v\rangle\simeq \langle
v,u\rangle$. In fact, by \eqref{e.ucheck}, we have commutative squares with
invertible $2$-morphisms
\[\xymatrix{&X\otimes X\spcheck \ar[rd]^{u\otimes \id}&& Y\otimes Y\spcheck \ar[rd]^{\ev_Y}\\
1_\cC\ar[ur]^{\coev_X}\ar[dr]_{\coev_Y} && Y\otimes X\spcheck\ar[rd]_{v\otimes \id}\ar[ru]^{\id\otimes v\spcheck} && 1_\cC.\\
& Y\otimes Y\spcheck\ar[ru]_{\id\otimes u\spcheck} && X\otimes X\spcheck\ar[ru]_{\ev_X}}
\]

The definition and construction above holds in particular for symmetric
monoidal $1$-categories. In the next subsection, $2$-morphisms will play an
important role.

\subsection{Functoriality of pairings}\label{ss.2}
A morphism $f\colon X\to X'$ in a $2$-category is said to be \emph{right
adjointable} if there exist a morphism $f^!\colon X'\to X$, called the right
adjoint of $f$, and $2$-morphisms $\eta\colon \id_X\to f^!\circ f$ and
$\epsilon\colon f\circ f^!\to \id_{X'}$ such that the composites
\[f\xrightarrow{\id \circ \eta} f\circ f^!\circ f \xrightarrow{\epsilon\circ \id} f,\quad f^!\xrightarrow{\eta\circ \id} f^!\circ f\circ f^!\xrightarrow{\id\circ \epsilon} f^!\]
are identities.

Let $(\cC,\otimes,1_{\cC})$ be a symmetric monoidal $2$-category.

\begin{construction}\label{c.fun}
Consider a diagram in $\cC$
\begin{equation}\label{e.twosquare}
\xymatrix{X\ar[r]^u\ar[d]_f & Y\ar[r]^v \ar[d]^g & X\ar[d]^f\\
X'\ar[r]^{u'} & Y'\ar[r]^{v'}\ultwocell\omit{^\alpha} & X'\ultwocell\omit{^\beta}}
\end{equation}
with $X$ and $X'$ dualizable and $f$ right adjointable. We will construct a
morphism $\langle u,v\rangle\to \langle u',v'\rangle$ in $\Omega \cC$.

In the case where $Y$ and $Y'$ are also dualizable and $g$ is also right
adjointable, we define $\langle u,v\rangle\to \langle u',v'\rangle$ by the
diagram
\[\xymatrix{1_\cC\ar[r]^{\coev_X}\ar[rd]_{\coev_{X'}}\drtwocell\omit{<-1>} & X\otimes X\spcheck \ar[rr]^{u\otimes v\spcheck}\ar[d]^{f\otimes {f^!}\spcheck}\drrtwocell\omit{}\ar@{}[rrd]|(.4){\alpha\otimes {\beta^!}\spcheck} && Y\otimes Y\spcheck\ar[rd]^{\ev_Y}\ar[d]_{g\otimes {g^!}\spcheck}\drtwocell\omit{<1>}\\
& X'\otimes {X'}\spcheck \ar[rr]^{u'\otimes {v'}\spcheck} && Y'\otimes {Y'}\spcheck \ar[r]^{\ev_{Y'}} &1_\cC}
\]
where $\beta^!$ is the composite
\[v\circ g^!\xrightarrow{\eta_f} f^!\circ
f\circ v\circ g^!\xrightarrow{\id\circ \beta\circ \id} f^!\circ v'\circ
g\circ g^!\xrightarrow{\epsilon_g} f^!\circ v',
\]
and the $2$-morphisms in the triangles are
\begin{gather}
\label{e.barepsilon}(f\otimes {f^!}\spcheck )\circ \coev_X \simeq ((f\circ f^!)\otimes \id)\circ \coev_{X'} \xrightarrow{\epsilon_f}\coev_{X'},\\
\label{e.bareta}\ev_Y\xrightarrow{\eta_g} \ev_Y\circ ((g^!\circ g)\otimes \id)\simeq \ev_{Y'} \circ (g\otimes {g^!}\spcheck).
\end{gather}
In particular, a morphism $\tr(e)\to \tr(e')$ is defined for every diagram
in $\cC$ of the form
\begin{equation}\label{e.trace2}
\xymatrix{X\ar[r]^e\ar[d]_f & X\ar[d]^f\\
X'\ar[r]^{e'} & X'\ultwocell\omit{^{}}}
\end{equation}
with $X$ and $X'$ dualizable and $f$ right adjointable.

In general, we define $\langle u,v\rangle\to \langle u',v'\rangle$ as the
morphism $\tr(v\circ u)\to \tr(v'\circ u')$ associated to the composite
down-square of \eqref{e.twosquare}.
\end{construction}

Trace can be made into a functor $\End(\cC)\to \Omega\cC$, where $\End(\cC)$
is a $(2,1)$-category whose objects are pairs $(X,e\colon X\to X)$ with $X$
dualizable and morphisms are diagrams \eqref{e.trace2} with $f$ right
adjointable \cite[Section 2.1]{HSS}. Composition in $\End(\cC)$ is given by
vertical composition of diagrams.

For the case of Theorem \ref{t.lv0} where $f$ is not proper, we will need to
relax the adjointability condition in Construction \ref{c.fun} as follows.
In a $2$-category, a \emph{down-square equipped with a splitting} is a
diagram
\begin{equation}\label{e.ds}
\xymatrix{X\ar[r]^u_{\Downarrow}\ar[d]_{f} & Y\ar[d]^{g}\\
X'\ar[r]_{u'}^{\Downarrow}\ar[ru]|w & Y'.}
\end{equation}
Note that the composition of \eqref{e.ds} with a down-square on the left or
on the right is a down-square equipped with a splitting. Moreover, a
down-square with one vertical arrow $f$ right adjointable is equipped with a
splitting induced by the diagram
\[\xymatrix{X\ar[d]_f\ar@{=}[r]_{\eta\Downarrow} & X\ar[d]^f \\
X'\ar@{=}[r]^{\Downarrow\epsilon}\ar[ru]|{f^!} & X'.}
\]

\begin{construction}\label{c.fun2}
Consider a diagram in $\cC$
\begin{equation}\label{e.fun2}
\xymatrix{X\ar[r]^u_{\gamma\Downarrow}\ar[d]_f & Y\ar[r]^v \ar[d]^g & X\ar[d]^f\\
X'\ar[r]_{u'}^{\Downarrow\delta}\ar[ru]|w & Y'\ar[r]_{v'} & X'\ultwocell\omit{^\beta}}
\end{equation}
with $X$ and $X'$ dualizable. We will construct a
morphism $\langle u,v\rangle\to \langle u',v'\rangle$ in $\Omega \cC$.

In the case where $Y$ is also dualizable, we decompose \eqref{e.fun2} into
\[\xymatrix{X\ar[r]^u\ar[d]_f & Y\ar[r]^v\ar@{=}[d]  & X\ar[d]^f\\
X'\ar[r]^w\ar@{=}[d] & Y\ar[d]^g\ar[r]^{fv}\ultwocell\omit{^\gamma} & X' \ar@{=}[d]\ar@{}[ul]|=\\
X'\ar[r]_{u'} & Y'\ar[r]_{v'}\ultwocell\omit{^\delta} & X'\ultwocell\omit{^\beta}}
\]
and take the composite
\[\langle u,v\rangle\simeq \langle v,u\rangle \to \langle fv,w\rangle \simeq \langle w,fv\rangle \to \langle u',v'\rangle.\]
Here the two arrows are given by the case $f=\id$ of Construction
\ref{c.fun}. In particular, a morphism $\tr(e)\to \tr(e')$ is defined for
every diagram in $\cC$ of the form
\[
\xymatrix{X\ar[r]^e_{\Downarrow}\ar[d]_{f} & X\ar[d]^{f}\\
X'\ar[r]_{e'}^{\Downarrow}\ar[ru] & X'}
\]
with $X$ and $X'$ dualizable.

In general, we define $\langle u,v\rangle \to \langle u',v'\rangle$ as the
morphism $\tr(v\circ u)\to \tr(v'\circ u')$ associated to the horizontal
composition of \eqref{e.fun2}.
\end{construction}

\begin{remark}\label{r.functor}
Let $\cC$ and $\cD$ be symmetric monoidal $2$-categories and let $F\colon
\cC\to \cD$ be a symmetric monoidal functor. Then $F$ preserves duals,
pairings, and functoriality of pairings.
\end{remark}

\subsection{The Grothendieck construction}\label{ss.3}
Given a category $B$ and a (pseudo) functor $F\colon B\to \Cat$,
Grothendieck constructed a category cofibered over $B$, whose strict fiber
at an object $X$ of $B$ is $F(X)$ \cite[VI]{SGA1}. We review Grothendieck's
construction in the context of symmetric monoidal $2$-categories. Our
convention on $2$-morphisms is made with applications to categorical
correspondences in mind.

Let $(\cB,\otimes,1_\cB)$ be a symmetric monoidal $2$-category. We consider
the symmetric monoidal $2$-category $(\Cat^\co,\times,*)$, where $\Cat^\co$
denotes the $2$-category obtained from the 2-category $\Cat$ of categories
by reversing the $2$-morphisms, $\times$ denotes the strict product, and $*$
denotes the category with a unique object and a unique morphism.

\begin{construction}\label{c.Groth}
Let $F\colon (\cB,\otimes,1_\cB)\to (\Cat^\co,\times,*)$ be a right-lax
symmetric monoidal functor. We have an object $e_{F}$ of $F(1_\cB)$ and
functors $F(X)\times F(X')\xrightarrow{\boxtimes} F(X\otimes X')$ for
objects $X$ and $X'$ of $\cB$. Given morphisms $c\colon X\to Y$ and
$c'\colon X'\to Y'$ in $\cB$, we have a natural transformation
\begin{equation}\label{e.lax}
\xymatrix{F(X)\times F(X')\ar[r]^\boxtimes \ar[d]_{F(c)\times F(c')}\drtwocell\omit{}\ar@{}[rd]|(0.35){F_{c,c'}} & F(X\otimes X')\ar[d]^{F(c\otimes c')}\\
F(Y)\times F(Y')\ar[r]^\boxtimes & F(Y\otimes Y').}
\end{equation}
The Grothendieck construction provides a symmetric monoidal $2$-category
$(\cC,\otimes,1_\cC)$ as follows.

An object of $\cC=\cC_F$ is a pair $(X,L)$, where $X\in \cB$ and $L\in
F(X)$. A morphism $(X,L)\to (Y,M)$ in $\cC$ is a pair $(c,u)$, where
$c\colon X\to Y$ is a morphism in $\cB$ and $u\colon F(c)(L)\to M$ is a
morphism in $F(Y)$. A $2$-morphism $(c,u)\to (d,v)$ is a $2$-morphism
$p\colon c\to d$ such that the following diagram commutes:
\[\xymatrix{F(c)(L)\ar[r]^u &M.\\
F(d)(L)\ar[u]^{F(p)(L)}\ar[ru]_v}
\]
We take $1_\cC=(1_\cB,e_{F})$. We put $(X,L)\otimes (X',L')\colonequals
(X\otimes X',L\boxtimes L')$. For morphisms $(c,u)\colon (X,L)\to (Y,M)$ and
$(c',u')\colon (X',L')\to (Y',M')$, we put $(c,u)\otimes (c',u')\colonequals
(c\otimes c',v)$, where
\[v\colon F(c\otimes c')(L\boxtimes L')\xrightarrow{F_{c,c'}}
F(c)L\boxtimes F(c')L'\xrightarrow{u\boxtimes u'} M\boxtimes M'.\]
\end{construction}

In applications in later sections, $F_{c,c'}$ will be a natural isomorphism.

Given a morphism $f\colon X\to X'$ in $\cB$ and an object $L$ of $F(X)$, we
write $f_\natural=(f,\id_{F(f)L})\colon (X,L)\to (X',F(f)L)$.

\begin{lemma}\label{l.push}
Given a $2$-morphism
\begin{equation}\label{e.square}
\xymatrix{X\ar[r]^{c}\ar[d]_{f} & Y \ar[d]^{g} \\
X'\ar[r]^{c'} & Y'\ultwocell\omit{^{p}}}
\end{equation}
in $\cB$ and a morphism $(c,u)\colon (X,L)\to (Y,M)$ in $\cC$ above $c$,
there exists a unique morphism $(c',u')\colon (X',F(f)L)\to (Y',F(g)M)$ in
$\cC$ above $c'$ such that $p$ defines a $2$-morphism in $\cC$:
\[\xymatrix{(X,L)\ar[r]^{(c,u)}\ar[d]_{f_\natural} & (Y,M) \ar[d]^{g_\natural} \\
(X',F(f)L)\ar[r]^{(c',u')} & (Y',F(g)M).\ultwocell\omit{^{p}}}
\]
\end{lemma}

\begin{proof}
By definition, $u'$ is the morphism $F(c')F(f)L\simeq
F(c'f)L\xrightarrow{F(p)}F(gc)L\simeq F(g)F(c)L\xrightarrow{u} F(g)M$.
\end{proof}

\begin{remark}\label{r.adjoint}
Let $f\colon X\to X'$ be a morphism in $\cB$ admitting a right adjoint
$f^!\colon X'\to X$. Let $\eta\colon \id_X\to f^!\circ f$ and
$\epsilon\colon f\circ f^!\to \id_{X'}$ denote the unit and the counit. Let
$L$ be an object of $F(X)$.
\begin{enumerate}
\item $f_\natural\colon (X,L)\to (X',F(f)L)$ admits the right adjoint
\[f^\natural=(f^!,F(\eta)(L))\colon (X',F(f)L)\to (X,L),\]
with unit and counit given by $\eta$ and $\epsilon$.
\item Assume that $F_{c,c'}$ is an isomorphism for all $c$ and $c'$,
    $(X,L)$ is dualizable in $\cC$ of dual $(X\spcheck, L\spcheck)$, and
    $X'$ is dualizable in $\cB$ of dual ${X'}\spcheck$. Then
    $(X',F(f)(L))$ is dualizable in $\cC$ of dual
    $({X'}\spcheck,F({f^!}\spcheck)(L\spcheck))$. The coevaluation and
    evaluation are given by
\begin{multline*}
F(\coev_{X'})(e_F)\xrightarrow{F(\bar\epsilon)} F(f\otimes {f^!}\spcheck)F(\coev_{X})(e_F)\xrightarrow{\coev_L} F(f\otimes {f^!}\spcheck)(L\boxtimes L\spcheck)\\
\shoveright{\xrightarrow{F_{f,{f^!}\spcheck}}F(f)(L)\boxtimes F({f^!}\spcheck)(L\spcheck),}\\
\shoveleft{F(\ev_{X'})(F({f^!}\spcheck)(L\spcheck)\boxtimes F(f)(L))\xrightarrow{F_{{f^!}\spcheck,f}^{-1}} F(\ev_{X'})F( {f^!}\spcheck\otimes f)(L\spcheck\boxtimes L)}\\
\xrightarrow{F(\bar\eta)} F(\ev_X)(L\spcheck\boxtimes L) \xrightarrow{\ev_L} e_F,
\end{multline*}
where $\bar \epsilon$ is \eqref{e.barepsilon}, $\bar\eta$ is
\eqref{e.bareta} (with $g=f$), and $\coev_L$ and $\ev_L$ denote the second
components of $\coev_{(X,L)}$ and $\ev_{(X,L)}$, respectively.
\end{enumerate}
\end{remark}

\begin{construction}\label{c.Groth2}
Let $F,G\colon (\cB,\otimes,1_\cB)\to (\Cat^\co,\times,*)$ be right-lax
symmetric monoidal functors. Let $\alpha\colon F\to G$ be a right-lax
symmetric monoidal natural transformation, which consists of the following
data:
\begin{itemize}
\item For every object $X$ of $\cB$, a functor $\alpha_X\colon F(X)\to
    G(X)$;

\item For every morphism $c\colon X\to Y$, a natural transformation
\[\xymatrix{F(X)\ar[r]^{F(c)}\ar[d]_{\alpha_X}\drtwocell\omit{^{\alpha_c}} &F(Y)\ar[d]^{\alpha_Y}\\ G(X)\ar[r]^{G(c)} & G(Y);}\]

\item A morphism $e_\alpha\colon e_{G}\to \alpha_{1_\cB}(e_{F})$ in
    $F(1_\cB)$;

\item For objects $X$ and $X'$ of $\cB$, a natural transformation
\[\xymatrix{F(X)\times F(X')\ar[r]^\boxtimes \ar[d]_{\alpha_X\times \alpha_{X'}}\drtwocell\omit{^{}}\ar@{}[rd]|(0.4){\alpha_{X,X'}} & F(X\otimes X')\ar[d]^{\alpha_{X\otimes X'}}\\
G(X)\times G(X')\ar[r]^\boxtimes & G(X\otimes X')}
\]
\end{itemize}
subject to various compatibilities. We construct a right-lax symmetric
monoidal functor $\psi\colon (\cC_F,\otimes,1)\to (\cC_G,\otimes,1)$ as
follows.

We take $\psi(X,L)=(X,\alpha_X(L))$ and $\psi(c,u)=(c,\psi u)$, where
\[\psi u\colon G(c)(\alpha_X(L))\xrightarrow{\alpha_c} \alpha_Y(F(c)
L)\xrightarrow{u}\alpha_Y(M)
\]
for $(c,u)\colon (X,L)\to (Y,M)$. We let $\psi$ send every $2$-morphism $p$
to $p$. The right-lax symmetric monoidal structure on $\psi$ is given by
\begin{gather*}
(\id,e_\alpha)\colon
(1_\cB,e_{G})\to (1_\cB,\alpha_{1_\cB}(e_F))=\psi(1_\cB,e_F),\\
\psi(X,L)\otimes \psi(X',L')=(X\otimes X',\alpha_X(L)\boxtimes \alpha_{X'}(L'))\qquad\\
\qquad\xrightarrow{(\id,\alpha_{X,X'})}(X\otimes X',\alpha_{X\otimes X'}(L\boxtimes L'))=\psi((X,L)\otimes (X',L')),\\
\xymatrix{\psi(X,L)\otimes \psi(X',L')\ar[r]^{(\id,\alpha_{X,X'})} \ar[d]_{\psi(c,u)\otimes \psi(c',u')}\ar@{}[rd]|= & \psi((X,L)\otimes (X',L'))\ar[d]^{\psi((c,u)\otimes (c',u'))}\\
\psi(Y,M)\otimes \psi(Y',M')\ar[r]^{(\id,\alpha_{Y,Y'})} & \psi((Y,M)\otimes (Y',M')).}
\end{gather*}
This is a symmetric monoidal structure if $e_\alpha$ and $\alpha_{X,X'}$ are
isomorphisms (which will be the case in our applications).
\end{construction}

\begin{lemma}\label{l.push2}
Consider a $2$-morphism \eqref{e.square} in $\cB$ and a morphism
$(c,u)\colon (X,L)\to (Y,M)$ in $\cC$ above $c$. Let $(c',u')\colon
(X',F(f)L)\to (Y',G(g)M)$ be the morphism associated to $(c,u)$ and let
$(c',(\psi u)')\colon (X',G(f)\alpha_X L)\to (Y',G(g)\alpha_Y M)$ be the
morphism associated to $(c,\psi u)$. Then the following square commutes:
\[\xymatrix{G(c')G(f)\alpha_X L\ar[d]_{\alpha_f}\ar[r]^{(\psi u)'} & G(g)\alpha_Y M\ar[d]^{\alpha_g}\\
G(c')\alpha_{X'}F(f) L\ar[r]^{\psi u'} & \alpha_Y F(g) M.}
\]
\end{lemma}

\begin{proof}
The square decomposes into
\[\xymatrix{G(c')G(f)\alpha_X L\ar[r]^{G(p)}\ar[d]_{\alpha_{f}} & G(g)G(c)\alpha_X L\ar[d]^{\alpha_c}\\
G(c')\alpha_{X'}F(f)L\ar[d]_{\alpha_{c'}} & G(g)\alpha_Y F(c)L\ar[r]^{u}\ar[d]^{\alpha_g} &G(g)\alpha_Y M\ar[d]^{\alpha_g}\\
\alpha_{Y'}F(c')F(f)L\ar[r]^{F(p)} & \alpha_{Y'}F(g)F(c)L\ar[r]^{u} & \alpha_Y F(g) M}
\]
where the inner cells commute.
\end{proof}

\begin{construction}
Let $(\cB,\otimes,1_\cB)\xrightarrow{H}
(\cB',\otimes,1_{\cB'})\xrightarrow{G}(\Cat^\co,\times,*)$ be right-lax
symmetric monoidal functors. Then we have an obvious right-lax symmetric
monoidal functor $\cC_{GH}\to \cC_G$ sending $(X,L)$ to $(HX,L)$, $(c,u)$ to
$(Hc,u)$, and every $2$-morphism $p$ to $Hp$. This is a symmetric monoidal
functor if $H$ is.
\end{construction}

\begin{construction}\label{c.tri}
Let
\[\xymatrix{(\cB,\otimes,1_\cB)\ar[dr]^{F}\ar[d]_{H}\drtwocell\omit{<1>_{\alpha}}\\
(\cB',\otimes,1_{\cB'})\ar[r]_G &(\Cat^\co,\times,*)}
\]
be a diagram of right-lax symmetric monoidal functors and right-lax
symmetric monoidal transformation. Combining the two preceding
constructions, we obtain right-lax symmetric monoidal functors $\cC_{F}\to
\cC_{GH}\to \cC_G$.
\end{construction}

\section{A relative Lefschetz-Verdier formula}\label{s.2}
We apply the formalism of duals and pairings to the symmetric monoidal
$2$-category of cohomological correspondences, which we define in Subsection
\ref{s.22}. We prove relative K\"unneth formulas in Subsection \ref{s.Kunn}
and use them to show the equivalence of dualizability and local acyclicity
(Theorem \ref{t.dual}) in Subsection \ref{s.23}. We prove the relative
Lefschetz-Verdier theorem for dualizable objects (Theorem \ref{t.lv}) in
Subsection \ref{s.24}. Together, the two theorems imply Theorem \ref{t.lv0}.
In Subsection \ref{s.25}, we prove that base change preserves duals
(Proposition \ref{p.pull}).

We will often drop the letters $L$ and $R$ from the notation of derived
functors.

\subsection{Relative K\"unneth formulas}\label{s.Kunn}
We extend some K\"unneth formulas over fields \cite[III 1.6, Proposition
1.7.4, (3.1.1)]{SGA5} to Noetherian base schemes under the assumption of
universal local acyclicity. Some special cases over a smooth scheme over a
perfect field were previously known \cite[Corollary 3.3, Proposition
3.5]{YZ}.

Let $S$ be a coherent scheme and let $\Lambda$ be a torsion commutative
ring. Let $X$ be a scheme over $S$. We let $D(X,\Lambda)$ denote the
unbounded derived category of the category of \'etale sheaves of
$\Lambda$-modules on~$X$. Following \cite[Th. finitude, D\'efinition
2.12]{SGA4d}, we say that $L\in D(X,\Lambda)$ is \emph{locally acyclic} over
$S$ if the canonical map $L_x\to R\Gamma(X_{(x)t},L)$ is an isomorphism for
every geometric point $x\to X$ and every algebraic geometric point $t\to
S_{(x)}$. Here $X_{(x)t}\colonequals X_{(x)}\times_{S_{(x)}} t$ denotes the
Milnor fiber. For $X$ of finite type over $S$, local acyclicity coincides
with strong local acyclicity \cite[Lemma 4.7]{LZdual}.

\begin{notation}
For $a_X\colon X\to S$ separated of finite type, we write
$K_{X/S}=a_X^!\Lambda_S$ and $D_{X/S}=R\cHom(-,K_X)$. Note that
$K_{S/S}=\Lambda_S$ is in general not an (absolute) dualizing complex.
\end{notation}

Assume in the rest of Subsection \ref{s.Kunn} that $S$ and $\Lambda$ are
Noetherian. We let $D_{\ft}(X,\Lambda)$ denote the full subcategory of
$D(X,\Lambda)$ consisting of complexes of finite tor-amplitude.

\begin{prop}\label{p.la}
Let $X',X,Y$ be schemes of finite type over $S$ and let $f\colon X\to X'$ be
a morphism over~$S$. Let $M\in D_{\ft}(Y,\Lambda)$ universally locally
acyclic over $S$, $L\in D^+(X,\Lambda)$. Then the canonical morphism
$f_*L\boxtimes_S M\to (f\times_S \id_Y)_*(L\boxtimes_S M)$ is an
isomorphism.
\end{prop}

This follows from \cite[Theorem 7.6.9]{Fu}. We recall the proof for
completeness.

\begin{proof}
By cohomological descent for a Zariski open cover, we may assume $f$
separated. By Nagata compactification, we are reduced to two cases: either
$f$ is proper, in which case we apply proper base change, or $f$ is an open
immersion, in which case we apply \cite[Th. finitude, App., Proposition
2.10]{SGA4d} (with $i=\id_{X'}$).
\end{proof}

In the rest of Subsection \ref{s.Kunn}, assume that $m\Lambda=0$ for some
integer $m$ invertible on $S$.

\begin{prop}\label{p.lau}
Let $X',X,Y$ be schemes of finite type over $S$ and let $f\colon X\to X'$ be
a separated morphism over $S$. Let $M\in D_{\ft}(Y,\Lambda)$ universally
locally acyclic over $S$, $L\in D^+(X',\Lambda)$. Then the canonical
morphism $f^!L\boxtimes_S M\to (f\times_S \id_Y)^!(L\boxtimes_S M)$ is an
isomorphism.
\end{prop}

The morphism is adjoint to
\[(f\times_S \id_Y)_! (f^!L\boxtimes_S M)\simeq f_!f^!L\boxtimes_S
M\xrightarrow{\adj \boxtimes_S \id_M} L\boxtimes_S M,
\]
where $\adj\colon f_!f^!L\to L$ denotes the adjunction.

\begin{proof}
We may assume that $f$ is smooth or a closed immersion. For $f$ smooth of
dimension $d$, $f^*(d)[2d]\simeq f^!$ and the assertion is clear. Assume
that $f$ is a closed immersion, and let $j$ be the complementary open
immersion. Let $f_Y=f\times_S \id_Y$ and $j_Y=j\times_S \id_Y$. Then we have
a morphism of distinguished triangles
\[\xymatrix{f^!L\boxtimes_S M \ar[r]\ar[d]_{\alpha}& f^*L\boxtimes_S M\ar[r]\ar[d]^{\simeq} & f^*j_*j^*L\boxtimes_S
M\ar[d]^{\beta}\ar[r]&\\
f_Y^!(L\boxtimes_S M)\ar[r] & f_Y^*(L\boxtimes_S M)\ar[r] & f_Y^*j_{Y*}j_Y^*(L\boxtimes_S M)\ar[r]&,}
\]
where $\beta$ is an isomorphism by Proposition \ref{p.la}. It follows that
$\alpha$ is an isomorphism.
\end{proof}

The following is a variant of \cite[Corollary 8.10]{Saito} and \cite[Theorem
6.8]{LZdual}. Here we do not require smoothness or regularity.

\begin{cor}\label{c.DKunn}
Let $X$ and $Y$ be schemes of finite type over $S$, with $X$ separated over
$S$. Let $M\in D_{\ft}(Y,\Lambda)$ universally locally acyclic over $S$.
Then the canonical morphism $K_{X/S}\boxtimes_S M \to p_Y^! M$ is an
isomorphism, where $p_Y\colon X\times_S Y\to Y$ is the projection.
\end{cor}

\begin{proof}
This is Proposition \ref{p.lau} applied to $X'=S$ and $L=\Lambda_S$.
\end{proof}

\begin{prop}\label{p.D}
Let $X$ and $Y$ be schemes of finite type over $S$, with $X$ separated over
$S$. Let $M\in D_{\ft}(Y,\Lambda)$ universally locally acyclic over $S$,
$L\in D^-_c(X,\Lambda)$. Then the canonical morphism $D_{X/S}L\boxtimes_S
M\to R\cHom(p_X^*L,p_Y^!M)$ is an isomorphism. Here $p_X\colon X\times_S
Y\to X$ and $p_Y\colon X\times_S Y\to Y$ are the projections.
\end{prop}

The morphism is adjoint to $(D_{X/S}L\otimes L)\boxtimes_S M\to
K_{X/S}\boxtimes_S M\to p_Y^!M$.

\begin{proof}
By \cite[IX Proposition 2.7]{SGA4}, we may assume $L=j_!\Lambda$ for
$j\colon U\to X$ \'etale with $U$ affine. Then the morphism can be
identified with
\[j_*D_{U/S}
\Lambda_U\boxtimes_S M\to j_{Y*}(D_{U/S}\Lambda_U\boxtimes M)\to j_{Y*}R\cHom(\Lambda_{U\times_S Y},j_Y^!
p_Y^!M)\simeq R\cHom(j_{Y!}\Lambda_{U\times_S Y},p_Y^!M),
\]
where $j_Y=j\times_S \id_Y\colon U\times_S Y\to X\times_S Y$. The first
arrow is an isomorphism by Proposition \ref{p.la}. The second arrow is an
isomorphism by Corollary \ref{c.DKunn}.
\end{proof}

\subsection{The category of cohomological correspondences}\label{s.22}
Let $S$ be a coherent scheme and let $\Lambda$ be a torsion commutative
ring.

\begin{construction}\label{c.C}
We define the $2$-category of cohomological correspondences
$\cC=\cC_{S,\Lambda}$ as follows. An object of $\cC$ is a pair $(X,L)$,
where $X$ is a scheme separated of finite type over $S$ and $L\in
D(X,\Lambda)$. A correspondence over $S$ is a pair of morphisms
$X\xleftarrow{\vecl{c}} C\xrightarrow{\vecr{c}} Y$ of schemes over $S$,
where $X$, $Y$, and $C$ are separated and of finite type over $S$. A
morphism $(X,L)\to (Y,M)$ in $\cC$ is a cohomological correspondence over
$S$, namely a pair $(c,u)$, where $c=(\vecl{c},\vecr{c})$ is a
correspondence over $S$ and $u\colon \vecl{c}^*L\to \vecr{c}^!M$ is a
morphism in $D(C,\Lambda)$. Given cohomological correspondences
$(X,L)\xrightarrow{(c,u)} (Y,M)\xrightarrow{(d,v)}(Z,N)$, the composite is
$(e,w)$, where $e$ is the composite correspondence given by the diagram
\begin{equation}\label{e.comp}
\xymatrix{&& C\times_Y D\ar[rd]^{\vecr{c}'}\ar[ld]_{\vecl{d}'}\\
&C\ar[ld]_{\vecl{c}}\ar[rd]^{\vecr{c}} && D\ar[ld]_{\vecl{d}}\ar[rd]^{\vecr{d}}\\
X&&Y&&Z,}
\end{equation}
and $w$ is given by the composite
\[\vecl{d}'^*\vecl{c}^*L\xrightarrow{u}\vecl{d}'^*\vecr{c}^!M\xrightarrow{\alpha}\vecr{c}'^!\vecl{d}^*M\xrightarrow{v}\vecr{c}'^!\vecr{d}^!N,\]
where $\alpha$ is adjoint to the base change isomorphism
$\vecr{c}'_!\vecl{d}'^*\simeq \vecl{d}^*\vecr{c}_!$. Given $(c,u)$ and
$(d,v)$ from $(X,L)$ to $(Y,M)$, a $2$-morphism $(c,u)\to (d,v)$ is a
\emph{proper} morphism of schemes $p\colon C\to D$ satisfying
$\vecl{d}p=\vecl{c}$ and $\vecr{d}p=\vecr{c}$ and such that $v$ is equal to
\[\vecl{d}^*L\xrightarrow{\adj} p_*p^*\vecl{d}^*L\simeq
p_!\vecl{c}^*L
\xrightarrow{u}p_!\vecr{c}^!M\simeq p_!p^!\vecr{d}^!M\xrightarrow{\adj} \vecr{d}^!M.
\]
Here we used the canonical isomorphism $p_!\simeq p_*$. Composition of
$2$-morphisms is given by composition of morphisms of schemes.

The $2$-category admits a symmetric monoidal structure. We put $(X,L)\otimes
(X',L')\colonequals (X\times_S X',L\boxtimes_S L')$. Given $(c,u)\colon
(X,L)\to (Y,M)$ and $(c',u')\colon (X',L')\to (Y',M')$, we define
$(c,u)\otimes (c',u')$ to be $(d,v)$, where $d=(\vecl{c}\times_S
\vecl{c'},\vecr{c}\times_S \vecr{c'})$ and $v$ is the composite
\[\vecl{d}^* (L\boxtimes_S L')\simeq
\vecl{c}^*L\boxtimes_S\vecl{c'}^*L'\xrightarrow{u\boxtimes_S
u'}\vecr{c}^!M\boxtimes_S\vecr{c'}^!M'\xrightarrow{\alpha}
\vecr{d}^!(M\boxtimes_S M'),
\]
where $\alpha$ is adjoint to the K\"unneth formula $\vecr{d}_!(-\boxtimes_S
-)\simeq \vecr{c}_!-\boxtimes_S \vecr{c'}_!-$. Tensor product of
$2$-morphisms is given by product of morphisms of schemes over $S$. The
monoidal unit of $\cC$ is $(S,\Lambda_S)$.
\end{construction}

\begin{remark}
Let $\cB_S$ be the symmetric monoidal $2$-category of correspondences
obtained by omitting $L$ from the above construction. The symmetric monoidal
structure on $\cB_S$ is given by fiber product of schemes over $S$ (which is
not the product in $\cB_S$ for $S$ nonempty). Consider the functor $F\colon
\cB_S\to \Cat^\co$ carrying $X$ to $D(X,\Lambda)$ and
$c=(\vecl{c},\vecr{c})$ to $\vecr{c}_!\vecl{c}^*$, and a $2$-morphism
$p\colon c\to d$ to the natural transformation
$\vecr{d}_!\vecl{d}^*\xrightarrow{\adj} \vecr{d}_!p_*p^*\vecl{d}^*\simeq
\vecr{c}_!\vecl{c}^*$. The compatibility of $F$ with composition
\eqref{e.comp} is given by the base change isomorphism
$\vecl{d}^*\vecr{c}_!\simeq\vecr{c}'_!\vecl{d}'^*$. The functor $F$ admits a
right-lax symmetric monoidal structure given by $e_F=\Lambda_S$ and
$\boxtimes_S$, with K\"unneth formula for $!$-pushforward providing a
natural isomorphism $F_{c,c'}$ \eqref{e.lax}. The Grothendieck construction
(Construction \ref{c.Groth}) then produces $\cC_{S,\Lambda}$.
\end{remark}

The category $\Omega\cC$ consists of pairs $(X,\alpha)$, where $X$ is a
scheme separated of finite type over $S$ and $\alpha\in H^0(X,K_{X/S})$. A
morphism $(X,\alpha)\to (Y,\beta)$ is  a proper morphism $X\to Y$ of schemes
over $S$ such that $\beta=p_*\alpha$, where
\begin{equation}\label{e.pp}
p_*\colon H^0(X,K_{X/S})\to H^0(Y,K_{Y/S})
\end{equation}
is given by adjunction $p_*p^!\simeq p_!p^!\to \id$.

\begin{lemma}\label{l.Hom}
The symmetric monoidal structure $\otimes$ on $\cC$ is closed, with internal
mapping object $\cHom((X,L),(Y,M))=(X\times_S Y,R\cHom(p_X^*L,p_Y^!M))$.
\end{lemma}

\begin{proof}
We construct an isomorphism of categories
\[F\colon \Hom((X,L)\otimes (Y,M),(Z,N))\simeq \Hom((X,L),\cHom((Y,M),(Z,N)))\]
as follows. An object of the source (resp.\ target) is a pair
$(C\xrightarrow{c} X\times_S Y\times_S Z,u)$, where $u$ belongs to
$H^0(C,c^!-)$ applied to left-hand (resp.\ right-hand) side of the
isomorphism
\[\alpha\colon R\cHom(p_X^*L\otimes p_Y^*M,p_Z^!N)\simeq R\cHom(p_X^*L, R\cHom(p_Y^*M,p_Z^!N)).\]
Here $p_X,p_Y,p_Z$ denote the projections from $X\times_S Y\times_S Z$. We
define $F$ by $F(c,u)=(c,u')$, where $u'$ is the image of $u$ under the map
induced by $\alpha$, and $F(p)=p$ for every morphism $p$ in the source of
$F$.
\end{proof}

For an object $(X,L)$ of $\cC$ and a morphism $f\colon X\to X'$ of schemes
separated of finite type over $S$, we let
\[f_\natural=(\id_X,f)_\natural=((\id_X,f),L\xrightarrow{\adj} f^!f_!L)\colon
(X,L)\to (X',f_!L).\]

\begin{lemma}\label{l.adj}
Let $(X,L)$ be an object of $\cC$ and let $f\colon X\to X'$ be a proper
morphism of schemes separated of finite type over $S$. Then
$f_\natural\colon (X,L)\to (X',f_*L)$ admits the right adjoint
$f^\natural=((f,\id_{X}),f^*f_*L\xrightarrow{\adj} L)\colon (X',f_*L)\to
(X,L)$.
\end{lemma}

\begin{proof}
The counit $f_\natural f^\natural\to \id_{(X',f_*L)}$ is given by $f$ and
the unit $\id_{(X,L)}\to f^\natural f_\natural$ is given by the diagonal
$X\to X\times_{X'} X$. (This is an example of Remark \ref{r.adjoint} (a).)
\end{proof}

\begin{construction}[$!$-pushforward]\label{c.push}
Consider a commutative diagram of schemes separated of finite type over $S$
\begin{equation}\label{e.diag}
\xymatrix{X\ar[d]_{f} & C\ar[l]_{\vecl{c}}\ar[d]^{p}\ar[r]^{\vecr{c}} & Y\ar[d]^{g} \\
X' & C'\ar[l]_{\vecl{c'}}\ar[r]^{\vecr{c'}} & Y' }
\end{equation}
such that $q\colon C\to X\times_{X'} C'$ is proper. Let $(c,u)\colon
(X,L)\to (Y,M)$ be a cohomological correspondence above $c$. Let
$p^\sharp=(f,p,g)$. By Lemma \ref{l.push}, we have a unique cohomological
correspondence $(c',p^\sharp_!u)\colon (X',f_!L')\to (Y',g_!M')$ above $c'$
such that $q$ defines a $2$-morphism in $\cC$:
\[\xymatrix{(X,L)\ar[r]^{(c,u)}\ar[d]_{f_\natural} & (Y,M) \ar[d]^{g_\natural} \\
(X',f_!L)\ar[r]^{(c',p^\sharp_!u)} & (Y',g_!M).\ultwocell\omit{^{q}}}
\]
For a more explicit construction of $p^\sharp_!u$, see \cite[Construction
7.16]{six}. We will often be interested in the case where $f$, $g$, and $p$
are proper. In this case we write $p^\sharp_*u$ for $p^\sharp_!u$.

This construction is compatible with horizontal and vertical compositions.
\end{construction}

\subsection{Dualizable objects}\label{s.23}
Let $S$ and $\Lambda$ be as in Subsection \ref{s.22}. Next we study
dualizable objects of $\cC=\cC_{S,\Lambda}$.

\begin{prop}\label{p.dualizable}
Let $(X,L)$ be a dualizable object of $\cC$.
\begin{enumerate}
\item The dual of $(X,L)$ is $(X,D_{X/S} L)$ and the biduality morphism
    $L\to D_{X/S}D_{X/S} L$ is an isomorphism. Moreover, for any object
    $(Y,M)$ of $\cC$, the canonical morphisms
\begin{gather}
\label{e.DL} D_{X/S}L\boxtimes_S M\to R\cHom(p_X^*L,p_Y^!M),\\
\notag L\boxtimes_S D_{Y/S}M\to R\cHom(p_Y^*M,p_X^!L),\\
\notag D_{X/S} L\boxtimes_S D_{Y/S} M\to D_{X\times_S Y/S}(L\boxtimes_S M)
\end{gather}
are isomorphisms. Here $p_X\colon X\times_X Y\to X$ and $p_Y\colon
X\times_S Y\to Y$ are the projections.
\item For every morphism of schemes $g\colon Y\to Y'$ separated of finite
    type over $S$ and all $M\in D(Y,\Lambda)$, $M'\in D(Y',\Lambda)$, the
    canonical morphisms
    \begin{gather*}
    L\boxtimes_S g_*M\to (\id_X\times_S g)_* (L\boxtimes_S M),\\
    L\boxtimes_S g^! M'\to (\id_X\times_S g)^! (L\boxtimes_S M')
    \end{gather*}
    are isomorphisms. Moreover, for morphisms of schemes $f\colon X\to X'$
and $f'\colon X''\to X$ separated of finite type over $S$ such that
$(X',f_!D_{X/S}L)$ and $(X'',f'^*D_{X/S}L)$ are dualizable, and $M\in
D(Y,\Lambda)$, the
    canonical morphisms
    \begin{gather*}
    f_*L\boxtimes_S M\to (f\times_S \id_Y)_* (L\boxtimes_S M),\\
    f'^!L\boxtimes_S M\to (f'\times_S \id_Y)^! (L\boxtimes_S M)
    \end{gather*}
    are isomorphisms.
\item If $L\in D^+(X,\Lambda)$, then $L$ is locally acyclic over $S$.
\item If $R\Delta^!$ commutes with small direct sums and $U$ has finite
    $\Lambda$-cohomological dimension for every affine scheme $U$ \'etale
    over $X$, then $L$ is $c$-perfect. Here $\Delta\colon X\to X\times_S
    X$ is the diagonal.
\end{enumerate}
\end{prop}

Following \cite[XVII D\'efinition 7.7.1]{ILO} we say $L\in D(X,\Lambda)$ is
\emph{$c$-perfect} if there exists a finite stratification $(X_i)$ of $X$ by
constructible subschemes such that for each $i$, $L|_{X_i}\in
D(X_i,\Lambda)$ is locally constant of perfect values. For $\Lambda$
Noetherian, ``$c$-perfect'' is equivalent to ``$\in D_\cft$''.

The condition that $R\Delta^!$ commutes with small direct sums is satisfied
if
\begin{itemize}
\item[(*)] $S$ is Noetherian finite-dimensional and $m\Lambda=0$ with $m$
    invertible on~$S$,
\end{itemize}
by Lemma \ref{l.directsum} below and \cite[XVIII\textsubscript{A} Corollary
1.4]{ILO}. Moreover, the proof below shows that the assumption $L\in
D^+(X,\Lambda)$ in (c) can be removed under condition (*).

\begin{proof}
(a) follows from Remarks \ref{r.dual}, \ref{r.dual2} and the identification
of internal mapping objects (Lemma \ref{l.Hom}). Via biduality and
\eqref{e.DL}, the morphisms in (b) can be identified with the isomorphisms
\begin{gather*}
R\cHom(p'^*_XL\spcheck,p_{Y'}^!g_*M)\simeq R\cHom(p'^*_XL\spcheck,g_{X*}p_Y^!M)\simeq g_{X*} R\cHom(p_X^*L\spcheck,p_Y^!M),\\
R\cHom(p^*_XL\spcheck,p_{Y}^!g^!M')\simeq R\cHom(g_X^* p'^*_XL\spcheck,g_{X}^!p_{Y'}^!M')\simeq g_{X}^! R\cHom(p'^*_XL\spcheck,p_{Y'}^!M'),\\
R\cHom(p^*_{X'}f_!L\spcheck,p'^!_{Y}M)\simeq R\cHom(f_{Y!}p^*_XL\spcheck,p'^!_{Y}M)\simeq f_{Y*} R\cHom(p_X^*L\spcheck,p_Y^!M),\\
R\cHom(p^*_{X''}f'^*L\spcheck,p''^!_{Y}M')\simeq  R\cHom(f'^*_Yp_X^*L\spcheck,f'^!_Yp_Y^!M)\simeq f'^!_Y R\cHom(p_X^*L\spcheck,p_Y^!M),
\end{gather*}
where $L\spcheck=D_{X/S}L$, $g_X=\id_X\times_S g$, $f_Y=f\times_S \id_Y$,
$f'_{Y}=f'\times_S \id_Y$, and $p'_X\colon X\times_S Y'\to X$, $p'_Y\colon
X'\times_S Y\to X'$, $p''_Y\colon X''\times_S Y\to X''$ are the projection.
(c) follows from the first isomorphism in (b) and Lemma \ref{l.Saito} below.
For (d), note that for $M\in D(X,\Lambda)$, $\Hom(\Lambda_X,\Delta^!
(D_{X/S} L\boxtimes_S M))\simeq \Hom(L,M)$ by \eqref{e.DL}. Since $\Delta^!$
commutes with small direct sums and $\Lambda_X$ is a compact object of
$D(X,\Lambda)$, it follows that $L$ is a compact object, which is equivalent
to being $c$-perfect by \cite[Proposition 6.4.8]{BS}.
\end{proof}

The following is a variant of \cite[Theorem 7.6.9]{Fu} and \cite[Proposition
8.11]{Saito}.

\begin{lemma}\label{l.Saito}
Let $X\to S$ be a morphism of coherent schemes and let $L\in D(X,\Lambda)$.
Assume that for every quasi-finite morphism $g\colon Y\to Y'$ of affine
schemes with $Y'$ \'etale over $S$, the canonical morphism $L\boxtimes_S
g_*\Lambda_Y\to (\id_X\times_S g)_* (L\boxtimes_S \Lambda_Y)$ is an
isomorphism. Assume either $L\in D^+(X,\Lambda)$ or that $(\id_X\times_S
g)_*$ has bounded $\Lambda$-cohomological dimension. Then $L$ is locally
acyclic over $S$.
\end{lemma}

\begin{proof}
Let $s\to S$ be a geometric point and let $g\colon t\to S_{(s)}$ be an
algebraic geometric point. Consider the diagram
\[\xymatrix{X_t\ar[r]^{g_X}\ar[d] & X_{(s)}\ar[d] & X_s\ar[l]_{i_X}\ar[d]\\
t\ar[r]^g & S_{(s)} & s\ar[l]_i}
\]
obtained by base change. By the assumption and passing to the limit, the
morphism $L|_{X_s}\to i_X^*g_{X*}(L|_{X_t})$ can be identified with
$L\boxtimes_S-$ applied to $\Lambda_s\to i^*g_*\Lambda_t$, which is an
isomorphism.
\end{proof}

\begin{lemma}\label{l.directsum}
Let $i\colon Y\to X$ be a closed immersion of finite presentation. Assume
that $i^!$ has finite $\Lambda$-cohomological dimension, then $Ri^!$
commutes with small direct sums.
\end{lemma}

\begin{proof}
Let $j$ be the complementary open immersion. It suffices to show that $Rj_*$
commutes with small direct sums under the condition that $j_*$ has finite
$\Lambda$-cohomological dimension. This is standard. See for example
\cite[Lemma 1.10]{LZdual}.
\end{proof}

\begin{lemma}\label{l.Hom1}
An object $(X,L)$ of $\cC$ is dualizable if and only if the canonical
morphism $L\boxtimes_S D_{X/S} L\to R\cHom(p_2^*L,p_1^!L)$ is an
isomorphism. Here $p_1$ and $p_2$ are the projections $X\times_S X\to X$.
\end{lemma}

\begin{proof}
The ``only if'' part is a special case of Proposition \ref{p.dualizable}
(a). The ``if'' part follows from Lemma \ref{l.Hom0} and the identification
of the internal mapping objects (Lemma \ref{l.Hom}).
\end{proof}

\begin{remark}
The evaluation and coevaluation maps for a dualizable object $(X,L)$ of
$\cC$ can be given explicitly as follows. The evaluation map $(X\times_S X,
D_{X/S} L\boxtimes_S L)\to (S,\Lambda)$ is given by $X\times_S
X\xleftarrow{\Delta} X\to S$ and the usual evaluation map
\[\Delta^*(D_{X/S}
L\boxtimes_S L)\simeq D_{X/S} L \otimes L\to K_{X/S},\]
where $\Delta$
denotes the diagonal. The coevaluation map $(S,\Lambda)\to (X\times_S X,
L\boxtimes_S D_{X/S} L)$ is given by $S\leftarrow
X\xrightarrow{\Delta}X\times_S X$ and $\id_L$ considered as a morphism
\[\Lambda_X\to R\cHom(L,L)\simeq \Delta^!R\cHom(p_2^*L,p_1^!L)\simeq
\Delta^!(L\boxtimes_S D_{X/S} L).
\]
\end{remark}

We can identify dualizable objects of $\cC$ under mild assumptions.

\begin{theorem}\label{t.dual}
Let $S$ be a Noetherian scheme, $\Lambda$ a Noetherian commutative ring with
$m\Lambda=0$ for $m$ invertible on $S$. Let $X$ be a scheme separated of
finite type over $S$, $L\in D_{\cft}(X,\Lambda)$. Then $(X,L)$ is a
dualizable object of $\cC$ if and only if $L$ is locally acyclic over $S$.
In this case the dual of $(X,L)$ is $(X,D_{X/S}L)$.
\end{theorem}

We will use Gabber's theorem that for $X$ of finite type over $S$, $L\in
D^b_c(X,\Lambda)$ is locally acyclic if and only if it is universally
locally acyclic \cite[Corollary 6.6]{LZdual}.

\begin{proof}
We have already seen the last assertion and the ``only if'' part of the
first assertion in Parts (a) and (c) of Proposition \ref{p.dualizable}. The
``if'' part of the first assertion follows from Lemma \ref{l.Hom1},
Proposition \ref{p.D}, and Gabber's theorem.
\end{proof}

\begin{remark}
Without invoking Gabber's theorem, our proof and Proposition \ref{p.pull}
show that for $L\in D_\cft(X,L)$, $(X,L)$ is dualizable if and only if $L$
is universally locally acyclic over $S$.
\end{remark}

\begin{cor}\label{c.la}
For $S$, $\Lambda$, and $X$ as in Theorem \ref{t.dual} and $L\in
D_{\cft}(X,\Lambda)$ locally acyclic over $S$, $D_{X/S} L$ is locally
acyclic over $S$.
\end{cor}

This was known under the additional assumption that $S$ is regular (and
excellent) \cite[Corollary 5.13]{LZdual} (see also \cite[B.6 2)]{BG} for $S$
smooth over a field). Our proof here is different from the one in
\cite{LZdual}. In fact, without invoking Gabber's theorem, our proof here
shows that $D_{X/S}$ preserves universal local acyclicity and makes no use
of oriented topoi.

\begin{proof}
By Theorem \ref{t.dual} and Remark \ref{r.dual}, $(X,D_{X/S}L)$ is
dualizable. We conclude by Proposition \ref{p.dualizable} (c).
\end{proof}

\begin{cor}\label{c.field}
Let $S$ be an Artinian scheme, $\Lambda$ and $X$ as in Theorem \ref{t.dual},
and $L\in D(X,\Lambda)$. Then $(X,L)$ is a dualizable object of $\cC$ if and
only if $L\in D_{\cft}(X,\Lambda)$.
\end{cor}

\begin{proof}
For $L\in D_{\cft}(X,\Lambda)$, $L$ is locally acyclic over $S$ by
\cite[Th.\ finitude, Corollaire 2.16]{SGA4d} and thus $(X,L)$ is dualizable
by the theorem. (Alternatively one can apply Lemma \ref{l.Hom1} and
\cite[III (3.1.1)]{SGA5}.) For the converse, we may assume that $S$ is the
spectrum of a separably closed field by Proposition \ref{p.pull}. In this
case, Proposition \ref{p.dualizable} (d) applies.
\end{proof}

\subsection{The relative Lefschetz-Verdier pairing}\label{s.24}
Let $S$ be a coherent scheme and $\Lambda$ a torsion commutative ring.

\begin{notation}\label{n.p}
For objects $(X,L)$ and $(Y,M)$ of $\cC$ with $(X,L)$ dualizable and
morphisms $(c,u)\colon (X,L)\to (Y,M)$ and $(d,v)\colon (Y,M)\to (X,L)$, we
write the pairing $\langle (c,u),(d,v)\rangle\in \Omega \cC$ in Construction
\ref{c.pair} as $(F,\langle u,v\rangle)$, where $F=C\times_{X\times_S Y} D$.
We call $\langle u,v\rangle\in H^0(F,K_{F/S})$ the relative
Lefschetz-Verdier pairing. The pairing is symmetric: $\langle u,v\rangle$
can be identified with $\langle v,u\rangle$ via the canonical isomorphism
$\langle c,d\rangle\simeq \langle d,c\rangle$.

For an endomorphism $(e,w)$ of a dualizable object $(X,L)$ of $\cC$, we
write $\tr(e,w)=(X^e,\tr(w))$, where $X^e=E\times_{e,X\times_S X,\Delta} X$
and $\tr(w)=\langle w,\id_L \rangle\in H^0(X^e,K_{X^e/S})$. We define the
\emph{characteristic class} $\cc_{X/S}(L)$ to be
$\tr(\id_L)=\langle\id_L,\id_L\rangle\in H^0(X,K_{X/S})$. In other words,
$\dim(X,L)=(X,\cc_{X/S}(L))$.
\end{notation}

\begin{theorem}[Relative Lefschetz-Verdier]\label{t.lv}
Let
\begin{equation}\label{e.lv}
\xymatrix{X\ar[d]_{f} & C\ar[l]_{\vecl{c}}\ar[d]^{p}\ar[r]^{\vecr{c}} & Y\ar[d]^{g} & D\ar[l]_{\vecl{d}}\ar[r]^{\vecr{d}}\ar[d]^{q} &
X\ar[d]^{f}\\
X' & C'\ar[l]_{\vecl{c'}}\ar[r]^{\vecr{c'}} & Y' & D'\ar[l]_{\vecl{d'}}\ar[r]^{\vecr{d'}} & X'}
\end{equation}
be a commutative diagram of schemes separated of finite type over $S$, with
$p$ and $D\to D'\times_{Y'} Y$ proper. Let $L\in D(X,\Lambda)$ such that
$(X,L)$ and $(X',f_!L)$ are dualizable objects of $\cC$. Let $M\in
D(Y,\Lambda)$, $u\colon \vecl{c}^*L\to \vecr{c}^! M$, $v\colon
\vecl{d}^*M\to \vecr{d}^! L$. Then $s\colon C\times_{X\times_S Y} D\to
C'\times_{X'\times_S Y'} D'$ is proper and
\[s_*\langle u,v\rangle = \langle p^\sharp_!u,q^\sharp_!v\rangle.\]
\end{theorem}

Combining this with Theorem \ref{t.dual}, we obtain Theorem \ref{t.lv0}.

\begin{proof}
By Construction \ref{c.push} applied to the right half of \eqref{e.lv} and
to the decomposition (which was used in the proof of \cite[Proposition
8.11]{six})
\[
\xymatrix{X\ar[d]_{f} & C\ar[l]_{\vecl{c}}\ar@{=}[d]\ar[r]^{\vecr{c}} & Y\ar@{=}[d]\\
X'\ar@{=}[d] & C\ar[l]_{f\vecl{c}}\ar[d]^{p}\ar[r]^{\vecr{c}} & Y\ar[d]^{g}\\
X' & C'\ar[l]_{\vecl{c'}}\ar[r]^{\vecr{c'}} & Y'}
\]
of the left half of \eqref{e.lv}, we get a diagram in $\cC$
\[\xymatrix{(X,L)\ar[r]^{(c,u)}_\Downarrow\ar[d]_{f_\natural} & (Y,M) \ar[d]^{g_\natural}\ar[r]^{(d,v)} & (X,L)\ar[d]^{f_\natural}\\
(X',f_!L)\ar[r]_{(c',p^\sharp_!u)}^\Downarrow\ar[ru]|{(e,w)} & (Y',g_!M)\ar[r]_{(d',q^\sharp_!v)} & (X',f_!L)\ultwocell\omit{^{}}}
\]
where $e=(f\vecl{c},\vecr{c})$ and $w=(f,\id_C,\id_Y)_!u$. By Construction
\ref{c.fun2}, we then get a morphism $(F,\langle u,v\rangle)\to (F',\langle
p^\sharp_!u,q^\sharp_! v\rangle)$ in $\Omega \cC$ given by $s\colon F\to
F'$.
\end{proof}

In the case where $f$ is proper, the dualizability of $(X',f_*L)$ follows
from that of $(X,L)$ by Proposition \ref{p.push} below. Moreover, in this
case, by Lemma \ref{l.adj}, $f_\natural$ is right adjointable and it
suffices in the above proof to apply the more direct Construction
\ref{c.fun} in place of Construction \ref{c.fun2}.

\begin{cor}
Let $f\colon X\to X'$ be a proper morphism of schemes separated of finite
type over $S$ and let $L\in D(X,\Lambda)$ such that $(X,L)$ is a dualizable
object of $\cC$. Then $f_*\cc_{X/S}(L)=\cc_{X'/S}(f_* L)$.
\end{cor}

\begin{proof}
This follows from Theorem \ref{t.lv} applied to $c=d=(\id_X,\id_X)$,
$c'=d'=(\id_{X'},\id_{X'})$ and $u=v=\id_L$.
\end{proof}

\begin{prop}\label{p.push}
Let $f\colon X\to Y$ be a proper morphism of schemes separated of finite
type over $S$. Let $(X,L)$ be a dualizable object of $\cC$. Then $(Y,f_*L)$
is dualizable.
\end{prop}

\begin{proof}
This follows formally from Remark \ref{r.adjoint} (b). We can also argue
using internal mapping objects as follows. By Proposition \ref{p.dualizable}
and Lemma \ref{l.Hom1}, the canonical morphism
\[\alpha\colon D_{X/S}L\boxtimes_S M\to R\cHom(p_X^*L,p_Z^! M)\]
is an isomorphism for every object $(Z,M)$ of $\cC$ and it suffices to show
that the canonical morphism
\[\beta\colon D_{Y/S} f_* L\boxtimes_S M\to R\cHom(q_Y^*f_*L,q_Z^!M)\]
is an isomorphism. Here $p_X,p_Z,q_Y,q_Z$ are the projections as shown in
the commutative diagram
\[\xymatrix{X\ar[d]_f & X\times_S Z\ar[d]_{f\times_S \id_Z}\ar[l]_{p_X}\ar[rd]^{p_Z}\\
Y & Y\times_S Z\ar[r]^{q_Z}\ar[l]_{q_Y} & Z.}
\]
Via the isomorphisms $D_{Y/S} f_* L\boxtimes_S M\simeq (f\times_S
\id_Z)_*(D_{X/S}L\boxtimes_S M)$ and
\[
R\cHom(q_Y^*f_*L,q_Z^!M)
\simeq R\cHom((f\times_S \id_Z)_*p_X^*L,q_Z^!M)\simeq (f\times_S \id_Z)_*R\cHom(p_X^*L,p_Z^!M),
\]
$\beta$ can be identified with $(f\times_S \id_Z)_*\alpha$.
\end{proof}

\begin{remark}
The relative Lefschetz-Verdier formula and the proof given above hold for
Artin stacks of finite type over an Artin stack $S$, with proper morphisms
replaced by a suitable class of morphisms equipped with canonical
isomorphisms $f_!\simeq f_*$ (such as proper representable morphisms). The
characteristic class lives in $H^0(I_{X/S},K_{I_{X/S}/S})$, where
$I_{X/S}=X\times_{\Delta,X\times_S X,\Delta} X$ is the inertia stack of $X$
over~$S$.
\end{remark}

Theorem \ref{t.lv} does not cover the twisted Lefschetz-Verdier formula in
\cite[A.2.19]{XZ}

\begin{remark}
Scholze remarked that arguments of this paper also apply in the \'etale
cohomology of diamonds and imply the equivalence between dualizability and
universal local acyclicity in this situation. This fact and applications are
discussed in his work with Fargues on the geometrization of the Langlands
correspondence \cite{FS}. In \cite[Section~4]{KW}, Hansen, Kaletha, and
Weinstein adapt our formalism and prove a Lefschetz-Verdier formula for
diamonds and v-stacks .
\end{remark}

\subsection{Base change and duals}\label{s.25}

We conclude this section with a result on the preservation of duals by base
change.

Let $g\colon S\to T$ be a morphism of coherent schemes and let $\Lambda$ be
a torsion commutative ring.

\begin{prop}\label{p.pull}
Let $(Y,M)$ be a dualizable object of $\cC_{T,\Lambda}$. Then $(Y_S,
g_Y^*M)$ is a dualizable object of $\cC_{S,\Lambda}$ and the canonical
morphism $g_Y^*D_{Y/T} M\to D_{Y_S/S} g_Y^* M$ is an isomorphism. Here
$Y_S=Y\times_T S$ and $g_Y\colon Y_S\to Y$ is the projection.
\end{prop}

We prove the proposition by constructing a symmetric monoidal functor
$g^*\colon \cC_{T,\Lambda}\to \cC_{S,\Lambda}$ as follows. We take
$g^*(Y,M)=(Y_S,g_Y^*M)$. For $(d,v)\colon (Y,M)\to (Z,N)$, we take
$g^*(d,v)=(d_S,v_S)$, here $d_S$ is the base change of $d$ by $g$ and $v_S$
is the composite
\[\vecl{d}^*_S g_Y^*M\simeq g_D^*\vecl{d}^* M\xrightarrow{g_D^* v} g_D^*\vecr{d}^! M\to \vecr{d}_S^!g_Z^* M,\]
where $D$ is the source of $\vecl{d}$ and $\vecr{d}$, $g_D$ and $g_Z$ are
defined similarly to $g_Y$. For every $2$-morphism $p$ of $\cC_{T,\Lambda}$,
we take $g^*(p)=p\times_T S$. The symmetric monoidal structure on $g^*$ is
obvious. Proposition \ref{p.pull} then follows from the fact that $g^*\colon
\cC_{T,\Lambda}\to \cC_{S,\Lambda}$ preserves duals (Remark
\ref{r.functor}).

The construction above is a special case of Construction \ref{c.tri}
(applied to $H\colon \cB_T\to \cB_S$ given by base change by $g$ and
$\alpha_Y$ given by $g_Y^*$).

\begin{cor}\label{c.dualbc}
Let $g\colon S\to T$ be a morphism of coherent schemes with $T$ Noetherian
and let $\Lambda$ be a Noetherian commutative ring with $m\Lambda=0$ for $m$
invertible on~$T$. Then for any scheme $Y$ separated of finite type over $T$
and any $M\in D_{\cft}(Y,\Lambda)$ locally acyclic over~$T$, the canonical
morphism $g_Y^*D_{Y/T} M\to D_{Y_S/S} g_Y^* M$ is an isomorphism. Here
$Y_S=Y\times_T S$ and $g_Y\colon Y_S\to Y$ is the projection.
\end{cor}

Note that the statement does not involve cohomological correspondences.

\begin{proof}
This follows from Proposition \ref{p.pull} and Theorem \ref{t.dual}.
\end{proof}

\section{Nearby cycles over Henselian valuation rings}\label{s.3}
Let $R$ be a Henselian valuation ring and let $S=\Spec(R)$. We do not assume
that the valuation is discrete. In other words, we do not assume $S$
Noetherian. Let $\eta$ be the generic point and let $s$ be the closed point.
Let $X$ be a scheme of finite type over $S$. Let $X_\eta=X\times_S \eta$,
$X_s=X\times_S s$. We consider the morphisms of topoi
\[X_\eta\xrightarrow{\aPsi_{X}} X \atimes_S \eta \xleftarrow{i_X}
X_s\atimes_S \eta\simeq X_s\ttimes_s \eta,
\]
where $\atimes$ denotes the oriented product of topoi \cite[XI]{ILO} and
$\ttimes$ denotes the fiber product of topoi. Let $\Lambda$ be a commutative
ring such that $m\Lambda=0$ for some $m$ invertible on $S$. We will study
the composite functor
\[\Psi_{X}\colon D(X_\eta,\Lambda)\xrightarrow{\aPsi_{X}} D(X\atimes_S \eta,\Lambda)\xrightarrow{i_X^*} D(X_s\ttimes_s \eta,\Lambda).\]

Let $\bar s$ be an algebraic geometric point above $s$ and let $\bar\eta\to
S_{(\bar s)}$ be an algebraic geometric point above~$\eta$. The restriction
of $\Psi_X L$ to $X_{\bar s}\simeq X_{\bar s}\ttimes_{\bar s} \bar \eta$ can
be identified with $(j_* L)|_{X_{\bar s}}$, where $j\colon X_{\bar \eta}\to
X_{(\bar s)}$, and was studied by Huber \cite[Section 4.2]{Huber}. We do not
need Huber's results in this paper.

In Subsection \ref{s.31}, we study the symmetric monoidal functor given by
$\Psi$ and cohomological correspondences. We deduce that $\Psi$ commutes
with duals (Corollary \ref{c.Gabber}), generalizing a theorem of Gabber. We
also obtain a new proof of the theorems of Deligne and Huber that $\Psi$
preserves constructibility (Corollary \ref{c.fin}). In Subsection
\ref{s.32}, extending results of Vidal, we use the compatibility of
specialization with proper pushforward to deduce a fixed point result.

\subsection{K\"unneth formulas and duals}\label{s.31}

\begin{prop}[K\"unneth formulas]\label{p.PsiKunn}
Let $X$ and $Y$ be schemes of finite type over $S$ and let $L\in
D(X_\eta,\Lambda)$, $M\in D(Y_\eta,\Lambda)$, then the canonical morphisms
\[\aPsi_{X}
L\boxtimes \aPsi_{Y} M\to \aPsi_{X\times_S Y}(L\boxtimes M),\quad
\Psi_X
L\boxtimes \Psi_Y M\to \Psi_{X\times_S Y}(L\boxtimes M),
\]
are isomorphisms.
\end{prop}

The K\"unneth formula for $\Psi$ over a Henselian \emph{discrete} valuation
ring is a theorem of Gabber (\cite[Th\'eor\`eme 4.7]{Illusie}, \cite[Lemma
5.1.1]{BB}).

\begin{proof}
It suffices to show that the first morphism is an isomorphism. By passing to
the limit and the finiteness of cohomological dimensions, it suffices to
show that $\Psi_{X,U/S}\colon X_U\to X\atimes_S U$ satisfies K\"unneth
formula for each open subscheme $U\subseteq S$. We then reduce to the case
$U=S$, where the K\"unneth formula is \cite[Theorem A.3]{IZ}. The
$\Psi$-goodness is satisfied by Orgogozo's theorem (\cite[Th\'eor\`eme
2.1]{Org}, \cite[Example 4.26 (2)]{LZdual}).
\end{proof}

\begin{construction}
Let $f\colon X\to Y$ be a separated morphism of schemes of finite type over
$S$. Then we have canonical natural transformations
\begin{align}
f_s^*\Psi_{Y}\to \Psi_{X}f_\eta^*,\label{e.bc1}\\
\Psi_{Y}f_{\eta*}\to f_{s*}\Psi_{X},\label{e.bc2}\\
f_{s!}\Psi_{X}\to \Psi_{Y}f_{\eta!},\label{e.bc3}\\
\Psi_{Y}f_\eta^!\to f_s^!\Psi_{Y}.\label{e.bc4}
\end{align}
Here we denoted $f_s\ttimes_s \eta$ by $f_s$. \eqref{e.bc1} is the base
change
\[f_s^*i_Y^*\aPsi_{Y}\simeq i_X^*(f\atimes_S \id_\eta)^*\aPsi_Y \to i_X^*\aPsi_Xf_\eta^*\]
and \eqref{e.bc4} is defined similarly to \cite[(4.9)]{LZdual} as
\[i_X^*\aPsi_Xf_\eta^!\simeq i_X^*(f\atimes_S \id_\eta)^!\aPsi_Y \to f_s^!i_Y^*\aPsi_Y.\]
\eqref{e.bc1} and \eqref{e.bc2} correspond to each other by adjunction. The
same holds for \eqref{e.bc3} and \eqref{e.bc4}. For $f$ proper,
\eqref{e.bc2} and \eqref{e.bc3} are inverse to each other.
\end{construction}

\begin{construction}
We construct symmetric monoidal $2$-categories $\cC_1$ and $\cC_2$ and a
symmetric monoidal functor $\psi\colon \cC_1\to \cC_2$ as follows.

The construction of $\cC_1$ is identical to that of $\cC_{S,\Lambda}$
(Construction \ref{c.C}) except that we replace the derived category
$D(-,\Lambda)$ by $D((-)_\eta,\Lambda)$. Thus an object of $\cC_1$ is a pair
$(X,L)$, where $X$ is a scheme separated of finite type over $S$ and $L\in
D(X_\eta,\Lambda)$. A morphism $(X,L)\to (Y,M)$ is a pair $(c,u)$, where
$c\colon X\to Y$ is a correspondence over~$S$ and $(c_\eta,u)$ is a
cohomological correspondence over $\eta$. A $2$-morphism $(c,u)\to (d,v)$ is
a $2$-morphism $p\colon c\to d$ such that $p_\eta$ is a $2$-morphism
$(c_\eta,u)\to (d_\eta,v)$. We have $(X,L)\boxtimes (Y,M)= (X\times_S
Y,L\boxtimes_\eta M)$.  The monoidal unit is $(S,\Lambda_\eta)$.

The construction of $\cC_2$ is identical to that of $\cC_{s,\Lambda}$ except
that we replace the derived category $D(-,\Lambda)$ by $D((-)\ttimes_s
\eta,\Lambda)$. Thus an object of $\cC_2$ is a pair $(X,L)$, where $X$ is a
scheme separated of finite type over $s$ and $L\in D(X\ttimes_s
\eta,\Lambda)$. The monoidal unit is $(s,\Lambda_\eta)$.

We define $\psi$ by $\psi(X,L)=(X_s,\Psi_{X} L)$, $\psi(c,u)=(c_s,\psi u)$,
where $\psi u$ is specialization of $u$ defined as the composite
\[\vecl{c}_s^* \Psi_{X} L\xrightarrow{\eqref{e.bc1}} \Psi_{C} \vecl{c}_\eta^*L
\xrightarrow{\Psi_{C} (u)} \Psi_{C}\vecr{c}_\eta^! M \xrightarrow{\eqref{e.bc4}} \vecr{c}_s^!\Psi_{Y} M.
\]
For every $2$-morphism $p$, $\psi p=p_s$. The symmetric monoidal structure
is given by the K\"unneth formula (Proposition \ref{p.PsiKunn}) and the
canonical isomorphism $\Psi_{S} \Lambda_S \simeq \Lambda_\eta$.
\end{construction}

\begin{remark}\label{r.psi}
The symmetric monoidal $2$-category $\cC_1$ (resp.\ $\cC_2$) is obtained via
the Grothendieck construction (Construction \ref{c.Groth}) from the
right-lax symmetric monoidal functor $\cB_S\to \Cat^\co$ (resp.\ $\cB_s\to
\Cat^\co$) carrying $X$ to $D(X_\eta,\Lambda)$ (resp.\ $D(X\ttimes_s
\eta,\Lambda)$). The symmetric monoidal functor $\psi$ is a special case of
Construction \ref{c.tri} (with $H\colon \cB_S\to \cB_s$ given by taking
special fiber). More explicitly, if $\cC'_2$ denotes the symmetric monoidal
$2$-category obtained from the right-lax symmetric monoidal functor
$\cB_S\to \Cat^\co$ carrying $X$ to $D(X_s\ttimes_s \eta,\Lambda)$, then
$\psi$ decomposes into $\cC_1\xrightarrow{\psi_1}
\cC'_2\xrightarrow{\psi_2}\cC_2$, where $\psi_1$ carries $(X,L)$ to
$(X,\Psi_{X} L)$ and $\psi_2$ carries $(X,L)$ to $(X_s,L)$.
\end{remark}

The proof the following lemma is identical to that of Lemma \ref{l.Hom}.

\begin{lemma}\label{l.Hom2}
The symmetric monoidal structures $\otimes$ on $\cC_1$ (resp.\ $\cC_2$) is
closed, with mapping object $\cHom((X,L),(Y,M))=(X\times_S
Y,R\cHom(p_{X_\eta}^*L,p_{Y_\eta}^!M))$ (resp.\
$\cHom((X,L),(Y,M))=(X\times_s Y,R\cHom(p_X^*L,p_Y^!M))$).
\end{lemma}

\begin{remark}\label{r.psidual}
It follows from Remark \ref{r.dual2} and Lemma \ref{l.Hom2} that the dual of
a dualizable object $(X,L)$ in $\cC_1$ (resp.\ $\cC_2$) is $(X,D_{X_\eta}
L)$ (resp.\ $(X,D_{X\ttimes_s \eta} L)$). Here for $a\colon U\to \eta$ and
$b\colon V\to s$ separated of finite type, we write $K_U=K_{U/\eta}$,
$D_U=D_{U/\eta}$ and $K_{V\ttimes_s \eta}=(b\ttimes_s \eta)^!\Lambda_\eta$,
$D_{V\ttimes_s \eta}=R\cHom(-,K_{V\ttimes_s \eta})$.
\end{remark}

In the rest of Subsection \ref{s.31}, we assume that $\Lambda$ is
Noetherian.

\begin{prop}\label{p.psidual}
An object $(X,L)$ in $\cC_1$ or $\cC_2$ is dualizable if and only if $L\in
D_\cft$.
\end{prop}

\begin{proof}
By Lemma \ref{l.Hom0} and the identification of internal mapping objects
(Lemmas \ref{l.Hom} and \ref{l.Hom2}), $(X,L)$ in $\cC_1$ is dualizable if
and only if $(X_\eta,L)$ in $\cC_\eta$ is dualizable. The latter condition
is equivalent to $L\in D_\cft$ by Corollary \ref{c.field}.

Similarly, $(X,L)$ in $\cC_2$ is dualizable if and only if $(X_{\bar
s},L|_{X_{\bar s}})$ in $\cC_{\bar s}$ is dualizable, by \cite[Lemma
1.29]{LZdual}. The latter condition is equivalent to $L|_{X_{\bar s}}\in
D_\cft$, which is in turn equivalent to $L\in D_\cft$.
\end{proof}

\begin{cor}\label{c.Gabber}
Let $X$ be a scheme separated of finite type over $S$ and let $L\in
D^-_c(X_\eta,\Lambda)$. The canonical morphism $\Psi_X D_{X_\eta} L \to
D_{X_s\ttimes_s \eta}\Psi_X L$ is an isomorphism in $D(X_s\ttimes_s
\eta,\Lambda)$.
\end{cor}

This generalizes a theorem of Gabber for Henselian \emph{discrete} valuation
rings \cite[Th\'eor\`eme 4.2]{Illusie}. Our proof here is different from
that of Gabber. One can also deduce Corollary \ref{c.Gabber} from the
commutation of duality with sliced nearby cycles over general bases
\cite[Theorem 0.1]{LZdual}.

\begin{proof}
The cohomological dimension of $\Psi_X$ is bounded by $\dim(X_\eta)$. Thus
we may assume that $L$ is of the form $u_!\Lambda_U$, where $u\colon U\to
X_\eta$ is an \'etale morphism of finite type. In particular, we may assume
$L\in D_{\cft}(X_\eta,\Lambda)$. In this case, $(X,L)$ is dualizable by
Proposition \ref{p.psidual}. We conclude by the fact that $\psi$ preserve
duals (Remark \ref{r.functor}) and the identification of duals (Remark
\ref{r.psidual}).
\end{proof}

We also deduce a new proof of the following finiteness theorem of Deligne
(for Henselian \emph{discrete} valuation rings) \cite[Th.\ finitude,
Th\'eor\`eme 3.2]{SGA4d} and Huber \cite[Proposition 4.2.5]{Huber}. Our
proof relies on Deligne's theorem on local acyclicity over a field
\cite[Th.\ finitude, Corollaire 2.16]{SGA4d}.

\begin{cor}\label{c.fin}
Let $X$ be a scheme of finite type over $S$. Then $\Psi_X$ preserves $D^b_c$
and $D_\cft$.
\end{cor}

\begin{proof}
We may assume that $X$ is separated. As in the proof of Corollary
\ref{c.Gabber}, we are reduced to the case of $D_\cft$. This case follows
from Proposition \ref{p.psidual} and the fact that $\psi$ preserves
dualizable objects (Remark \ref{r.functor}).
\end{proof}

By Remark \ref{r.functor}, $\psi$ also preserves pairings, and we obtain the
following generalization of \cite[Proposition 1.3.5]{Var}.

\begin{cor}
Consider morphisms of schemes separated of finite type over $S$:
\[\xymatrix{X& C\ar[l]_{\vecl{c}}\ar[r]^{\vecr{c}} & Y & D\ar[l]_{\vecl{d}}\ar[r]^{\vecr{d}}& X.}\]
Let $L\in D_\cft(X_\eta,\Lambda)$, $M\in D(Y_\eta,\Lambda)$, $u\colon
\vecl{c}_\eta^*L\to \vecr{c}_\eta^!M$, $v\colon \vecl{d}_\eta^* M\to
\vecr{d}_\eta^! L$. Then $\spe\langle u,v\rangle=\langle \psi u,\psi
v\rangle$, where $\spe$ is the composition
\[H^0(F_\eta,K_{F_\eta})\to
H^0(F_s\ttimes_s \eta,\Psi_F K_{F_\eta})\to H^0(F_s\ttimes_s \eta,K_{F_s\ttimes_s \eta})
\]
and $F=C\times_{X\times Y}D$.
\end{cor}

\subsection{Pushforward and fixed points}\label{s.32}

\begin{construction}[$!$-Pushforward in $\cC_2$]
Consider a commutative diagram \eqref{e.diag} in  $\cB_s$ such that $q\colon
C\to X\times_{X'} C'$ is proper. Let $(c,u)\colon (X,L)\to (Y,M)$ be a
morphism in $\cC_2$ above $c$. By Lemma \ref{l.push}, we have a unique
morphism $(c',p^\sharp_{!}u)\colon (X',f_{!}L')\to (Y',g_{!}M')$ in $\cC_2$
above $c'$ such that $q$ defines a $2$-morphism in $\cC_2$:
\[\xymatrix{(X,L)\ar[r]^{(c,u)}\ar[d]_{f_\natural} & (Y,M) \ar[d]^{g_\natural} \\
(X',f_!L)\ar[r]^{(c',p^\sharp_!u)} & (Y',g_!M).\ultwocell\omit{^{q}}}
\]
For $f$, $g$, $p$ proper, we write $p^\sharp_*u$ for $p^\sharp_!u$.
\end{construction}

Applying Lemma \ref{l.push2} to the functor $\psi_1$ in Remark \ref{r.psi},
we obtain the following.

\begin{prop}\label{p.sp}
Consider a commutative diagram of schemes separated of finite type over $S$
\[
\xymatrix{X\ar[d]_{f} & C\ar[l]_{\vecl{c}}\ar[d]^{p}\ar[r]^{\vecr{c}} & Y\ar[d]^{g} \\
X' & C'\ar[l]_{\vecl{c'}}\ar[r]^{\vecr{c'}} & Y' }
\]
such that $C\to X\times_{X'} C'$ is proper. Let $L\in D(X_\eta,\Lambda)$,
$M\in D(Y_\eta,\Lambda)$, $u\colon \vecl{c}_\eta^* L\to \vecr{c}_\eta^! M$.
Then the square
\[\xymatrix{\vecl{c'}_s^*f_{s!}\Psi_X L\ar[r]^{p^\sharp_{s!} \psi u}\ar[d] & \vecr{c'}_s^!g_{s!}\Psi_Y M\ar[d]\\
\vecl{c'}_s^*\Psi_{X'} f_{\eta!}L\ar[r]^{\psi p^\sharp_{\eta!}u} & \vecr{c'}_s^!\Psi_{Y'} g_{\eta!}M}
\]
commutes. Here the vertical arrows are given by \eqref{e.bc3}. In
particular, in the case where $f$, $g$, $p$ are proper, $p^\sharp_{s*} \psi
u$ can be identified with $\psi p^\sharp_{\eta*}u$ via the isomorphisms
$f_{s*}\Psi_X\simeq \Psi_{X'} f_{\eta*}$ and $g_{s*}\Psi_{Y}\simeq \Psi_{Y'}
g_{\eta*}$.
\end{prop}

This generalizes a result of Vidal \cite[Th\'eor\`eme 7.5.1]{Vidal1} for
certain Henselian valuation rings of rank~$1$. As in \cite[Sections 7.5,
7.6]{Vidal1}, Proposition \ref{p.sp} implies the following fixed point
result, generalizing \cite[Proposition 5.1, Corollaire 7.5.3]{Vidal1}.

\begin{cor}\label{c.Vidal}
Assume that $\eta$ is separably closed. Consider a commutative diagram of
schemes
\[\xymatrix{X\ar[r]^f\ar[d]^g & S\ar[d]_\sigma^{\simeq}\\
X\ar[r]^f & S}
\]
with $f$ proper and $\sigma$ fixing $s$. Assume that $g_s$ does not fix any
point of $X_s$. Then $\tr(g,R\Gamma(X_\eta,\Lambda))=0$. If, moreover, $g$
is an isomorphism and $U\subseteq X_\eta$ is an open subscheme such that
$g(U)=U$, then $\tr(g,R\Gamma_c(U,\Lambda))=0$.
\end{cor}

\begin{proof}
For completeness, we recall the arguments of \cite[Corollaire
7.5.2]{Vidal1}. We may assume $\Lambda=\Z/m\Z$. We decompose the commutative
diagram into
\[\xymatrix{X\ar[r]^\gamma\ar[rd]_g &\sigma^*X\ar[r]\ar[d]^\sigma & S\ar[d]^\sigma\\
& X\ar[r]^f & S.}
\]
Consider the cohomological correspondences
$((\id_{X_s},\id_{X_s}),\sigma)\colon (X_s,\Psi_{X}\Lambda)\to
(X_s,\Psi_{\sigma^* X}\Lambda)$ and $(c_\eta,u)\colon
(\sigma^*X_\eta,\Lambda)\to (X_\eta,\Lambda)$, where $c=(\gamma,\id_{X})$
and $u=\id_{\Lambda_{X_\eta}}$. We have a commutative diagram
\[\xymatrix{R\Gamma(X_\eta,\Lambda)\ar[r]^\sigma\ar[d]_{\simeq} & R\Gamma(\sigma^*X_\eta,\Lambda)\ar[d]^\simeq\ar[r]^{f_{\eta*}^\sharp u} &
R\Gamma(X_\eta,\Lambda)\ar[d]^\simeq\\
R\Gamma(X_s,\Psi_X \Lambda)\ar[r]^{\sigma} & R\Gamma(X_s,\Psi_{\sigma^*X} \Lambda)\ar[r]^{f_{s*}^\sharp \psi u} & R\Gamma(X_s,\Psi_X \Lambda)}
\]
where the square on the right commutes by Proposition \ref{p.sp}. The
composite of the upper horizontal arrows is the action of $g$. Thus, by the
Lefschetz-Verdier formula over $s$, we have
\[\tr(g,R\Gamma(X_\eta,\Lambda))=\int_{X_s^{g_s}} \langle \sigma,\psi u\rangle=0,\]
where $\int_{F}\colon H^0(F,K_{F})\to \Lambda$ denotes the trace map. For
the last assertion of the corollary, it suffices to note that
\[\tr(g,R\Gamma_c(U,\Lambda))=\tr(g,R\Gamma(X_\eta,\Lambda))-\tr(g,R\Gamma(Z_\eta,\Lambda))=0,\]
where $Z$ is the closure of $X_\eta\backslash U$ in $X$, equipped with the
reduced subscheme structure.
\end{proof}

\subsection*{Financial support}
This work was partially supported by National Key Research and Development
Program of China (grant number 2020YFA0712600), National Natural Science
Foundation of China (grant numbers 12125107, 11822110, 11688101, 11621061);
Beijing Natural Science Foundation (grant number 1202014), Fundamental
Research Funds for Central Universities of China, and China Scholarship
Council.

\end{document}